\documentclass[12pt,leqno,a4paper]{article}
\usepackage{amsthm}
\usepackage{indentfirst}
\usepackage{amsmath}
\usepackage{mathrsfs}
\usepackage{paralist}
\usepackage{amssymb,enumerate}
\usepackage{hyperref}
\usepackage{framed}
\usepackage{extarrows}
\usepackage[hmargin=3cm,vmargin=3cm]{geometry}
\usepackage{url}

\theoremstyle{theorem}
\newtheorem{mainthm}{Theorem}

\swapnumbers
\theoremstyle{theorem}
\newtheorem{thm}{Theorem}[section]

\newtheorem{lem}[thm]{Lemma}

\theoremstyle{definition}

\newtheorem{rem}[thm]{Remark}
\newtheorem{thmwot}[thm]{}

\numberwithin{equation}{thm}

\newcommand{\Ind}{\operatorname{Ind}}

\newcommand{\oo}{\operatorname{o}}
\newcommand{\op}{\operatorname{op}}
\newcommand{\rank}{\operatorname{rank}}
\newcommand{\Res}{\operatorname{Res}}

\newcommand{\cO}{\mathcal{O}}

\newcommand{\Set}[1]{\left\{\,#1\,\right\}}
\newcommand{\lsup}[1]{{^{#1}\hspace{-0.8pt}}}
\newcommand{\lsub}[1]{{_{#1}\hspace{-0.8pt}}}

\title{On stable equivalences of Morita type and nilpotent blocks
\footnote{supported by Natural Science Foundation of Sichuan Province (No. 2024NSFJQ0070) and by National Natural Science Foundation of China (No. 12271446).}}
\author{Conghui Li\footnote{Conghui Li, School of Mathematics, Southwest Jiaotong University, Chengdu 611756, China, Email: liconghui@swjtu.edu.cn}}
\date{}

\begin{document}

\setlength\abovedisplayskip{1ex plus 0.2ex minus 0.1ex}
\setlength\belowdisplayskip{1ex plus 0.2ex minus 0.1ex}

\raggedbottom
\maketitle

\begin{abstract}
In this note, we give a new proof by module-theoretic methods for a result of Puig asserting that blocks which are stable equivalent of Morita type to nilpotent blocks are also nilpotent.


\end{abstract}

\section{Introduction}

Nilpotent blocks were introduced by M. Brou\'e and L. Puig \cite{BrouPuig80} to give a representation-theoretic analogue of Frobenius' result on $p$-nilpotent groups.
Puig gave a description of the source algebras of nilpotent blocks in \cite{Puig88}.
For an exposition of the theory of nilpotent blocks, see for example \cite[\S8.11]{Linck18b}.
Stable equivalences of Morita type was introduced by Brou\'e \cite{Brou94} after the similar results for derived equivalences by J. Rickard; for the definition and basic properties of stable equivalences of Morita type, see for example \cite[\S2.17, \S4.14]{Linck18a}.
In \cite{Puig99}, Puig proved that blocks which are stable equivalent of Morita type to nilpotent blocks are also nilpotent by showing that such stable equivalences are induced by bimodules with endopermutation sources.
For the definition and basic properties of endopermutation modules, see for example \cite[\S7.3]{Linck18b}.
This result answered affirmatively a question raised by Puig himself in \cite[1.8]{Puig88}.
The method in \cite{Puig99} heavily relies on the theory of $G$-algebras and pointed groups systematically developed by Puig; for an exposition of this theory, see for example \cite{Thev95}.
In this note, we give a new proof using module-theoretic methods for Puig's result.

Throughout this note, all groups are assumed to be finite; $\cO$ would always be a complete discrete valuation ring with quotient field $K$ of characteristic $0$ and residue field $k$ of prime characteristic $p$; blocks are always $p$-blocks; we assume $k$ is large enough for all blocks involved below.

\begin{mainthm}[{\cite[Theorem 8.2]{Puig99}}]\label{mainthm}
Let $G$ be a finite group and $b$ be a block of $\cO{G}$.
If $\cO{G}b$ is stable equivalent of Morita type to a nilpotent block, then this equivalence is induced by a bimodule with endopermutation sources, and thus $\cO{G}b$ is also nilpotent.
\end{mainthm}

After giving some notation, preliminaries and quoted results in \S\ref{prelim}, we give our proof of Theorem \ref{mainthm} in \S\ref{proof}.
As in \cite{Puig99}, the proof here also uses Weiss' criterion for permutation modules, which was proved in \cite{Weiss88} for unramified $\cO$ and was generalized to arbitrary $\cO$ by K. Roggenkamp \cite{Rogg92}; see also \cite[Theorem A1.2]{Puig99}.
Besides Weiss' criterion, two key results used in the proof are Puig's result on local fusion \cite{Puig86} (\emph{i.e.} the fusion systems can be determined by source algebras) for which we follow the statements and treatments in \cite[\S8.7]{Linck18b} using module-theoretic methods and Puig's result on stable equivalences of Morita type with endopermutation source \cite[7.6]{Puig99} which is stated and proved using module-theoretic methods in \cite[Theorem 9.11.2]{Linck18b}.
Other ingredients of the proof include a Mackey-like formula of Bouc \cite{Bouc10}.


\section{Notation, preliminaries and quoted results}\label{prelim}

\begin{thmwot}\label{notation}
In this note, any ring has an identity; all (left, right and bi-) modules are assumed to be finitely generated; left modules are abbreviated as modules.
The notation for representations of finite groups are standard; see for example \cite{Linck18a,Linck18b}.
Here we just fix some conventions and notation.
Let $R$ be an arbitrary commutative ring, $A,B$ be $R$-algebras and $G,H$ be groups.

For two $A$-modules $U,V$, the notation $U \mid V$ means $U$ is isomorphic to a direct summand of $V$.

We use the convention to view $(A,B)$-bimodules as $A\otimes_RB^{\op}$-modules; in particular, right (left) $B$-modules are viewed as left (right) $B^{\op}$-modules.
$(RG,RH)$-bimodules can be viewed as $R(G\times H)$-module via the group isomorphism $H \to H^{\op},\ h \mapsto h^{-1}$.
If furthermore $G=H$, an $R(G \times G)$-module also has an $RG$-module structure via the homomorphism $G \to G \times G,\ g \mapsto (g,g)$.
For example, for an $RG$-module $V$, the dual module $V^*$ has a right $RG$-module structure and a left $RG$-module structure.
Usually, one can determine which module structure is used by context.

Let $X$ be a subgroup of $G \times H$ and $V$ be an $RX$-module.
For any $(g,h)\in X$ and $v\in V$, we write the action of $(g,h)$ on $v$ as $(g,h).v=g.v.h^{-1}$.
The dual module $\left(\Ind_X^{G\times H}V\right)^*$ has both a left $R(G\times H)$-module structure and a right $R(G\times H)$-module structure.
The right $R(G \times H)$-module structure on $\left(\Ind_X^{G\times H}V\right)^*$ induces an $(RH,RG)$-bimodule structure, that induces again an $R(H \times G)$-module structure, for which there is an isomorphism
\begin{equation}\label{equ-IndXGHV*}
\left(\Ind_X^{G\times H}V\right)^* \cong \Ind_{X^\sharp}^{H\times G}V^*,
\end{equation}
where $X^\sharp = \Set{ (h,g) \mid (g,h)\in X }$ and $(h,g)\in X^\sharp$ acts on $f\in V^*$ as
\[ ((h,g).f)(v) = f(g^{-1}.v.h),\ \forall\, v\in V. \]

Finally, we fix some notation for direct products of groups and their subgroups.
Let $X$ be a subgroup of $G \times H$.
We denote by $\pi_1$ ($\pi_2$, resp.) the projection from $G \times H$ to the first component $G$ (the second component $H$, resp.), and we also denote by $\pi_1,\pi_2$ the restrictions of these two projections to $X$ by abuse of notation.
In the sequel, we identify $G$ and $H$ with $G \times 1$ and $1\times H$ respectively in $G\times H$ when no confusion will be produced.
The notation $\iota_1(X)$, when viewed as a subgroup of $G$, means the subgroup of $G$ such that $X \cap (G \times 1) = \iota_1(X) \times 1$, while when viewed as a subgroup of $G \times H$, $\iota_1(X)$ means $X \cap (G \times 1)$; similarly for $\iota_2(X)$.
We will leave the context to determine the exact meaning of such notation.
\end{thmwot}

We will use a Mackey-like formula for tensor products of bimodules by S. Bouc which we quoted as follows.

\begin{thm}[Bouc \cite{Bouc10}]\label{Bouc}
Let $G,H,K$ be groups and $X$ ($Y$, resp.) be a subgroup of $G \times H$ ($H \times K$, resp.).
Assume $U$ is an $RX$-module and $V$ is an $RY$-module.
Then there is an isomorphism of $(RG,RK)$-bimodules
\[ \left( \Ind_X^{G\times H} U \right) \otimes_{RH} \left( \Ind_Y^{H\times K} V \right) \cong \bigoplus_{t\in [\pi_2(X)\backslash{H}/\pi_1(Y)]} \Ind_{X\ast\lsup{(t,1)}Y}^{G\times K} U \otimes_{R[\iota_2(X)\cap\lsup{t}(\iota_1(Y))]} \lsup{(t,1)}V, \]
where
\[ X\ast\lsup{(t,1)}Y = \Set{ (g,k)\in G\times K \mid \exists\, h\in H,\ (g,h)\in X,\, (h^t,k)\in Y } \]
and $(g,k)\in X\ast\lsup{(t,1)}Y$ acts on $u\otimes v \in U \otimes_{R[\iota_2(X)\cap\lsup{t}(\iota_1(Y))]} \lsup{(t,1)}V$ as
\[ (g,k).(u\otimes v) = g.u.h^{-1} \otimes h^t.v.k^{-1} \]
if $h$ is chosen such that $(g,h)\in X$ and $(h^t,k)\in Y$. 
\end{thm}

\begin{thmwot}
By the completeness assumption of $\cO$, the blocks of $\cO{G}$ and $kG$ correspond bijectively, called the blocks of $G$.
We follow the routine to identify a block of $\cO{G}$ with its primitive central idempotent $b$ and call $\cO{G}b$ the associated block algebra.
R. Brauer defined an important invariant, namely defect groups, for blocks of $G$; see for example \cite[\S6.1, \S6.2]{Linck18b}.
J.L. Alperin and M. Brou\'e introduced Brauer pairs as local data for blocks of finite groups; see for example \cite[\S6.3]{Linck18b}.
L. Puig developed the theory of fusion systems which can be used to organize these local data into a small category.
For fusion systems in general, see for example \cite[Part I]{AschKesOliv11} or \cite[\S\S8.1--8.3]{Linck18b} (note that fusion systems in \cite{Linck18b} mean saturated fusion systems as in \cite{AschKesOliv11}); for fusion systems associated to blocks, see \cite[\S8.5]{Linck18b} or \cite[Part IV]{AschKesOliv11}.

Puig also introduced for each block the concept of source algebras, which are Morita equivalent to the associated block algebra.
For definition and basic properties of source algebras, see for example \cite[\S6.4]{Linck18b}.
Source algebras are better compatible with the local data than block algebras.
For example, by results of Puig \cite{Puig86}, the fusion systems can be determined by source algebras; we will follow the statements in \cite[\S8.7]{Linck18b}.
\end{thmwot}

We will use the following lemma in our later proof.

\begin{lem}\label{lem-rank-iM}
Let $G$ be a group, $b$ be a block of $\cO{G}$ with defect group $P$ and a source idempotent $i\in (\cO{G}b)^P$.
Assume $M$ is an $\cO$-free indecomposable $\cO{G}b$-module with a vertex $Q\leq P$ and an $\cO{Q}$-source $V$.
Then $|P:Q|$ divides $\rank_\cO iM$ and $\rank_\cO iM \geq |P:Q|\rank_\cO V$.
Furthermore, there is an indecomposable direct summand of $iM$ as $\cO{P}$-module whose $\cO$-rank is equal to $|P:Q|\rank_\cO V$ and this indecomposable direct summand is non-projective as an $\cO{P}$-module whenever $M$ is non-projective as an $\cO{G}$-module.
\end{lem}

\begin{proof}
By \cite[6.3]{Linck94} or \cite[6.4.10]{Linck18b}, we may assume $V \mid \Res^{i\cO{G}i}_{\cO{Q}}iM$ and $iM \mid i\cO{G}i \otimes_{\cO{Q}} V$.
It follows from the $(\cO{P},\cO{Q})$-bimodule structure of $i\cO{G}i$ (see \cite[Theorem 8.7.1]{Linck18b}) and Green's indecomposable theorem that $\Res^{i\cO{G}i}_{\cO{P}}iM$ is isomorphic to a direct sum of indecomposable modules of the form $\cO{P}\otimes_{\cO{S}}\lsub{\varphi}W$, where $S$ is a subgroup of $P$, $\varphi\colon S \to Q$ belongs to the fusion system of $\cO{G}b$ determined by $i$ (see \cite[p.198]{Linck18b}) and $W$ is an indecomposable direct summand of $\Res^Q_{\varphi(S)}V$.
So $|P:Q|$ divides $\rank_\cO iM$. 
Now it follows from $V \mid \Res^{i\cO{G}i}_{\cO{Q}}iM$ that there are some $S_0,\varphi_0,W_0$ as above such that $\cO{P}\otimes_{\cO{S_0}}\lsub{\varphi_0}W_0 \mid iM$ as $\cO{P}$-modules and $V \mid \Res^P_Q\cO{P}\otimes_{\cO{S_0}}\lsub{\varphi_0}W_0$.
Since $V$ has $Q$ as a vertex, $\varphi_0\colon S_0 \to Q$ must be an isomorphism and $W_0=V$.
Thus we have $\rank_\cO iM \geq \rank_\cO \cO{P}\otimes_{\cO{S_0}}\lsub{\varphi_0}V = |P:Q|\rank_\cO V$, and the equality shows also the last assertion.
\end{proof}

\begin{rem}
Keep the notation in Lemma \ref{lem-rank-iM}.
Similar results for $iM$ in Lemma \ref{lem-rank-iM} also holds for $M$.
\end{rem}

Finally, we record Weiss' criterion for later use.

\begin{thm}[Weiss, \cite{Weiss88}, \cite{Rogg92}, {\cite[Theorem A1.2]{Puig99}}]\label{Weiss}
Let $P$ be a finite $p$-group, $M$ be an $\cO$-free $\cO{P}$-module.
If there is a normal subgroup $Q$ of $P$ such that $\Res^P_QM$ is projective and $M^Q$ is a permutation $\cO(P/Q)$-module, then $M$ is a permutation $\cO{P}$-module.
\end{thm}

In the sequel, we will also use \cite[8.7.1, 9.11.2]{Linck18b}, but since they are well-presented in book form, we do not bother to include their statements here.

\section{Proof of the main result}\label{proof}

\begin{thmwot}\label{setting-SEMT}
We begin with some basic settings and results for stable equivalences of Morita type between blocks of finite groups.
Let $G$\,($H$, resp.) be a group and $b$\,($c$, resp.) be a block of $G$\,($H$, resp.) with a defect group $P$\,($Q$, resp.).
Assume $M$ is an indecomposable $(\cO{H}c,\cO{G}b)$-bimodule projective as left and right module inducing a stable equivalence of Morita type between $\cO{G}b$ and $\cO{H}c$.
We may assume that both $b$ and $c$ are not of defect zero.
As an $\cO(H\times G)$-module, $M$ belongs to the block $c\otimes b^{\oo}$ of $H \times G$ with $Q \times P$ as a defect group, where $b^{\oo}$ is the block of $\cO{G}$ corresponding to the block $b$ of $\cO{G}^{\op}$ induced by the group isomorphism $G \to G^{\op},\ g\mapsto g^{-1}$; see for example \cite[Proposition 8.7.7]{Linck18b}.
Thus we may choose a vertex $R$ of $M$ such that $ R \leq Q \times P$.

Fix an $\cO{R}$-source $V$ of $M$.
Then we have
\[ M \mid \cO{H}c \otimes_{\cO{Q}} \Ind_R^{Q \times P} V \otimes_{\cO{P}} \cO{G}b. \]
Furthermore, there are primitive idempotents $i\in(\cO{G}b)^P$ and $j\in(\cO{H}c)^Q$ such that
\begin{equation}
M \mid \cO{H}j \otimes_{\cO{Q}} \Ind_R^{Q \times P} V \otimes_{\cO{P}} i\cO{G}.
\end{equation}
We claim that $i$\,($j$, resp.) is a source idempotent of $\cO{G}b$\,($\cO{H}c$, resp.).
Assume $i$ belongs to a non-local point of $P$ on $\cO{G}b$.
Thus there is a proper subgroup $P_0$ of $P$ such that $i\in (\cO{G}b)_{P_0}^P$.
Then by \cite[Theorem 5.12.8]{Linck18a}, there is $i_0\in(\cO{G}b)^{P_0}$ such that $i\cO{G} \cong \cO{P} \otimes_{\cO{P_0}} i_0\cO{G}$.
So we have $M \mid \cO{H}j \otimes_{\cO{Q}} \Ind_R^{Q \times P} V \otimes_{\cO{P_0}} i_0\cO{G}$, and it follows from $\cO{G}b \mid M^* \otimes_{\cO{H}c} M$ that $\cO{G}b \mid \cO{G} \otimes_{\cO{P_0}} W \otimes_{\cO{P_0}} \cO{G}$ for some indecomposable $\cO(P_0 \times P_0)$-module $W$.
Consequently, as an $\cO(G\times G)$-module, $\cO{G}b$ has a vertex contained in $P_0\times P_0$, which contradicts to the fact that $\Delta{P}$ is a vertex of $\cO{G}b$.
Thus $i$ is a source idempotent of $\cO{G}b$; similarly for $j$.

We will need a slight generalization of part of arguments in \cite[Theorem 9.11.2]{Linck18b} as follows.

Since $\cO{G}i$ is an indecomposable non-projective $(\cO{G}b,\cO{P})$-bimodule and $M$ induces a stable equivalence, the $(\cO{H}c,\cO{P})$-bimodule $Mi \cong M \otimes_{\cO{G}b} \cO{G}i$ can be decomposed as
\begin{equation}\label{equ-Mi-N0-N1}
Mi = N_0 \oplus N_1,
\end{equation}
where $N_0$ is an indecomposable non-projective $(\cO{H}c,\cO{P})$-bimodule and $N_1$ is a projective $(\cO{H}c,\cO{P})$-bimodule.

It follows from (\ref{equ-Mi-N0-N1}) that
\begin{equation}
M^* \otimes_{\cO{H}c} Mi \cong (M^* \otimes_{\cO{H}c} N_0) \oplus (M^* \otimes_{\cO{H}c} N_1)
\end{equation}
as $(\cO{G}b,\cO{P})$-bimodules.
Noting that $M$ is projective as left and right module, $M^* \otimes_{\cO{H}c} N_1$ is projective as an $(\cO{G}b,\cO{P})$-bimodule.
Since $\cO{G}i \mid M^* \otimes_{\cO{H}c} Mi$ and $\cO{G}i$ is indecomposable non-projective as an $(\cO{G}b,\cO{P})$-bimodule, we have that
\begin{equation}\label{equ-OGi-mid-M*-N0}
\cO{G}i \mid M^* \otimes_{\cO{H}c} N_0.
\end{equation}

Since $N_0 \mid \cO{H}j \otimes_{\cO{Q}} \Ind_R^{Q \times P} V \otimes_{\cO{P}} i\cO{G}i$ as $(\cO{H}c,\cO{P})$-bimodules, there is an indecomposable direct summand $W_1$ of $i\cO{G}i$ as $\cO(P\times P)$-modules such that
\[ N_0 \mid \cO{H}j \otimes_{\cO{Q}} \Ind_R^{Q \times P} V \otimes_{\cO{P}} W_1. \]
Then it follows from (\ref{equ-OGi-mid-M*-N0}) that $\cO{G}i \mid \cO{G} \otimes_{\cO{P}} W_2 \otimes_{\cO{P}} W_1$ for some $(\cO{P},\cO{P})$-bimodule $W_2$.
Since the $\cO(G\times P)$-module $\cO{G}i$ has $\Delta{P}$ as a vertex, it follows from \cite[8.7.1]{Linck18b} that $W_1 \cong \cO[yP]$ for some $y \in N_G(P,e_P)$ with $(P,e_P)$ the Brauer pair determined by $i$.
Noting that $Mi \cong Miy^{-1}$, we have $N_0 \cong N_0y^{-1}$ and 
\begin{equation}
N_0 \mid \cO{H}j \otimes_{\cO{Q}} \Ind_R^{Q \times P} V
\end{equation}
and
\begin{equation}\label{equ-N*0-N0}
N^*_0 \otimes_{\cO{H}c} N_0 \mid \Ind_{R^\sharp}^{P\times Q} V^* \otimes_{\cO{Q}} j\cO{H}j \otimes_{\cO{Q}} \Ind_R^{Q\times P} V.
\end{equation}
Here we use the isomorphism (\ref{equ-IndXGHV*}).

Note that we have
\begin{equation}
i\cO{G}i \oplus X \cong iM^* \otimes_{\cO{H}c} Mi
\end{equation}
for some projective $(i\cO{G}i,i\cO{G}i)$-bimodule $X$.
Combined with (\ref{equ-Mi-N0-N1}), we have as $(\cO{P},\cO{P})$-bimodules that
\begin{equation}
i\cO{G}i \oplus X \cong iM^* \otimes_{\cO{H}c} Mi \cong (N^*_0 \otimes_{\cO{H}c} N_0) \oplus Y
\end{equation}
where $X$ and $Y$ are both projective as $\cO(P\times P)$-modules (for the projectivity of $Y$, we need to use the fact that $M$ is projective as left and right module and thus so are $N_0$ and $N_1$).
Thus we also have:
\begin{compactenum}[(\ref{setting-SEMT}.1)]\setcounter{enumi}{7}
\item\label{N^*0-N0-P-P-stable-basis}
$iM^* \otimes_{\cO{H}c} Mi$ and $N^*_0 \otimes_{\cO{H}c} N_0$ have $P \times P$-stable basis;
\item\label{item-OP-N*0-N0}
The non-projective indecomposable direct summands of $i\cO{G}i$ and of $N^*_0 \otimes_{\cO{H}c} N_0$ as $\cO(P\times P)$-modules are the same up to isomorphism; in particular, $\cO{P} \mid N^*_0 \otimes_{\cO{H}c} N_0$.
\end{compactenum}

\end{thmwot}

\begin{lem}[{\cite[Theorem 6.9]{Puig99}}]\label{lem-surj-proj}
Keep notation in \ref{setting-SEMT}, then the both projections $\pi_1\colon R \to Q$ and $\pi_2\colon R \to P$ are surjective.
\end{lem}

\begin{proof}
It follows from (\ref{equ-N*0-N0}) and (\ref{setting-SEMT}.\ref{item-OP-N*0-N0}) that as $\cO(P\times P)$-modules
\[ \cO{P} \mid \Ind_{R^\sharp}^{P\times Q} V^* \otimes_{\cO{Q}} W \otimes_{\cO{Q}} \Ind_R^{Q\times P} V, \]
where $W$ is an indecomposable direct summand of $j\cO{H}j$ as an $\cO(Q\times Q)$-modules.
By \cite[Theorem 8.7.1]{Linck18b}, we have that
\[ W \cong \cO{Q} \otimes_{\cO{S}} {_{\varphi}\cO{Q}} \]
for some $S \leq Q$ and some $\varphi\colon S \to Q$ in the fusion system of $\cO{H}c$ determined by $j$.
Thus we have
\[ \cO{P} \mid \Ind_{R^\sharp}^{P\times Q} V^* \otimes_{\cO{S}} \lsub{\varphi}\Ind_R^{Q\times P} V. \]
Using Mackey formula for $\Res^{P\times Q}_{P\times S}\Ind_{R^\sharp}^{P\times Q} V^*$ and $\Res^{Q\times P}_{\varphi(S)\times P}\Ind_R^{Q\times P} V$, we have that
\[ \cO{P} \mid \Ind_{R_1}^{P\times S} V_1 \otimes_{\cO{S}} \lsub{\varphi}\Ind_{R_2}^{\varphi(S)\times P} V_2, \]
where $R_1 \leq P \times S$,  $R_2 \leq \varphi(S) \times P$, $V_1$ is an $\cO{R}_1$-module and $V_2$ is an $\cO{R}_2$-module.
Here, it is readily to see that the elements appearing as the first components of elements of $R_1$ also appear as the first components of some elements of $R^\sharp$; similarly for $R_2$.
Since
\[ \lsub{\varphi}\Ind_{R_2}^{\varphi(S)\times P} V_2 \cong \Ind_{R_3}^{S\times P} V_3 \]
for some $R_3 \leq S \times P$ and some $\cO R_3$-module $V_3$, we have that
\[ \cO{P} \mid \Ind_{R_1}^{P\times S} V_1 \otimes_{\cO{S}} \Ind_{R_3}^{S\times P} V_3. \]
By Bouc's formula \ref{Bouc} and Green's indecomposable theorem, $\cO{P}$ is isomorphic to $\Ind_{R_4}^{P \times P}V_4$ for some $R_4 \leq P \times P$ and some indecomposable $\cO R_4$-module $V_4$.
Again, it is readily to see that the elements appearing as the first (second, resp.) components of elements of $R_4$ also appear as the first (second, resp.) components of some elements of $R_1$ ($R_3$, resp.).
Since as an $\cO(P\times P)$-module $\cO{P}$ has a vertex $\Delta{P}$, projections from $R_4$ to the first component and to the second component are both surjective, which implies that $\pi_2\colon R \to P$ is surjective by tracking through every step above.
Symmetric arguments show that $\pi_1\colon R \to Q$ is also surjective.
\end{proof}

Now, we give the proof of our main result as follows.

\begin{thmwot}[Proof of Theorem \ref{mainthm}]\label{proof-mainthm}
Keep the notation of \ref{setting-SEMT} and assume $\cO{H}c$ is nilpotent.
By \cite[7.6]{Puig99} (see also \cite[Theorem 9.11.2]{Linck18b}), stable equivalences of Morita type with endopermutation sources preserve fusion systems, thus it suffices to prove that the $\cO{R}$-source $V$ of the bimodule $M$ in \ref{setting-SEMT} is an endopermutation module.

By the structure of source algebras of nilpotent blocks in \cite{Puig88}, $\cO{H}c$ is Morita equivalent to $\cO{Q}$ via a bimodule with endopermutation source; see \cite[Theorem 8.11.9]{Linck18b} or \cite[Theorem 9.11.9]{Linck18b}.
Since stable equivalences of Morita type with endopermutation source are closed under compositions and taking inverses, we may therefore assume that $\cO{H}c=j\cO{H}j=\cO{Q}$.
Consequently, $jMi=Mi$, $jN_0=N_0$, $jN_1=N_1$ and $N_0$ is up to isomorphism the unique indecomposable non-projective direct summand of $Mi$ as $\cO(Q\times P)$-module.

Now (\ref{equ-N*0-N0}) becomes
\begin{equation}\label{equ-N*N-mid}
N^*_0 \otimes_{\cO{Q}} N_0 \mid \Ind_{R^\sharp}^{P\times Q} V^* \otimes_{\cO{Q}} \Ind_R^{Q\times P} V.
\end{equation}
With $G,\, b,\, P,\, i,\, M,\, Q,\, V$ replaced by $H\times G,\, c\otimes b^{\oo},\, Q\times P,\, j\otimes i^{\oo},\, M,\, R,\, V$ as in \ref{setting-SEMT}, it follows from Lemma \ref{lem-rank-iM} that
\[ \rank_\cO N_0 = |Q \times P : R|\rank_\cO V = \rank_\cO \Ind_R^{Q\times P} V. \]
Since $M$ is projective as a left module and $N_0 \mid M$, $N_0$ is projective as a left $\cO{Q}$-module and thus
\[ \rank_\cO N^*_0 \otimes_{\cO{Q}} N_0 = (\rank_\cO N_0)^2/|Q|; \]
again, since $M$ is projective as a left module and $V \mid M$ as left $\cO{R}$-modules, $\Res^R_{\iota_1(R)}V$ is projective and so is $\Res^{Q\times P}_{Q\times1}\Ind_R^{Q\times P} V \cong \Ind_{\iota_1(R)}^{Q\times1}V$ (here, we use the surjection of the projections $\pi_1\colon R \to Q$ and $\pi_2\colon R \to P$ from Lemma \ref{lem-surj-proj}), thus
\[ \rank_\cO \left( \Ind_{R^\sharp}^{P\times Q} V^* \otimes_{\cO{Q}} \Ind_R^{Q\times P} V \right) = (\rank_\cO \Ind_R^{Q\times P} V)^2/|Q|. \]
Consequently, there is an isomorphism of $\cO(P\times P)$-modules
\begin{equation}\label{equ-N*N-cong}
N^*_0 \otimes_{\cO{Q}} N_0 \cong \Ind_{R^\sharp}^{P\times Q} V^* \otimes_{\cO{Q}} \Ind_R^{Q\times P} V.
\end{equation}

Again using the surjection of the projections $\pi_1\colon R \to Q$ from Lemma \ref{lem-surj-proj}, Bouc' formula \ref{Bouc} gives that
\[ \Ind_{R^\sharp}^{P\times Q} V^* \otimes_{\cO{Q}} \Ind_R^{Q\times P} V \cong \Ind_{R^\sharp\ast R}^{P\times P} (V^* \otimes_{\cO[\iota_1(R)]} V). \]
Since $\Delta{P} \leq R^\sharp\ast R$ by the definition of $R^\sharp\ast R$ and the surjection of $\pi_2\colon R \to P$, it follows from Mackey formula that
\[ V^* \otimes_{\cO[\iota_1(R)]} V \mid \Res^{P\times P}_{\Delta{P}} \left( N^*_0 \otimes_{\cO{Q}} N_0 \right). \]
Thus it follows from (\ref{setting-SEMT}.\ref{N^*0-N0-P-P-stable-basis}) that $V^* \otimes_{\cO[\iota_1(R)]} V$ is a permutation $\cO{\Delta{P}}$-module.

To conclude, we use some arguments in the proof of \cite[Theorem 8.2]{Puig99}.
Since $\Res^R_{\iota_1(R)}V$ is projective, so is $\Res^R_{\iota_1(R)} (V^* \otimes_\cO V)$ and there is an isomorphism
\[ V^* \otimes_{\cO[\iota_1(R)]} V \cong (V^* \otimes_\cO V)^{\iota_1(R)}, \]
where, the action of $\Delta{P}$ on $V^* \otimes_{\cO[\iota_1(R)]} V$ via restriction of the action of $R^\sharp\ast R$ on $V^* \otimes_{\cO[\iota_1(R)]} V$ as in the Bouc's formula \ref{Bouc}, coincides with the action of $R/\iota_1(R)$ on $(V^*\otimes_{\cO}V)^{\iota_1(R)}$ via the isomorphism $\Delta{P} \cong P \cong R/\iota_1(R)$ (note that $\iota_1(R)$ is the kernel of the surjective map $\pi_2\colon R \to P$).
Thus we have that $(V^* \otimes_\cO V)^{\iota_1(R)}$ is a permutation $\cO(R/\iota_1(R))$-module since $V^* \otimes_{\cO[\iota_1(R)]} V$ is a permutation $\cO{\Delta{P}}$-module.
It follows from Weiss' criterion \ref{Weiss} that $V^*\otimes_{\cO}V$ is a permutation $\cO{R}$-module, \emph{i.e.} $V$ is an endompermutation $\cO{R}$-module, which completes the proof.
\end{thmwot}

\begin{rem}
We give an alternative proof of (\ref{equ-N*N-cong}) by obtaining some more precise information about $jN_0$.
Keep the notation of \ref{setting-SEMT}.

First we claim (without the assumption that $\cO{H}c$ is nilpotent) that
\begin{equation}\label{equ-Ind-mid-N0}
\Ind_R^{Q\times P} V \mid jN_0
\end{equation}
as $\cO(Q\times P)$-modules.
We use the same arguments in \cite[9.10.4]{Linck18b} or \cite[4.6]{Linck15} for our slightly more general cases.
By \cite[6.3]{Linck94} or \cite[6.4.10]{Linck18b}, $M$ has a vertex-source pair $(R_1,V_1)$ such that $R_1 \leq Q \times P$ and $V_1$ is a direct summand of $jMi$ as an $\cO{R_1}$-module.
Then it follows from \cite[5.1.6]{Linck18b} that as an $\cO(Q \times P)$-module $jMi$ has an indecomposable direct summand with vertex-source pair $(R_1,V_1)$.
Green's Indecomposable Theorem implies that $\Ind_{R_1}^{Q \times P}V_1 \mid jMi$ and thus
\[ \Ind_{R_1}^{Q \times P}V_1 \mid j\cO{H}j \otimes_{\cO{Q}} \Ind_R^{Q \times P}V \otimes_{\cO{P}} i\cO{G}i. \]
Using \cite[8.7.1]{Linck18b} and noting that both $\pi_1\colon R_1 \to Q$ and $\pi_2\colon R_1 \to P$ are surjective, we have that
\[ \Ind_{R_1}^{Q \times P}V_1 \cong \lsub{\psi}\left(\Ind_R^{Q \times P}V\right)_{\varphi}, \]
where $\psi\colon Q \to Q$ ($\varphi\colon P \to P$, resp.) belongs to the fusion system of $\cO{H}c$ ($\cO{G}b$, resp) determined by $j$ ($i$, resp.).
By \cite[8.7.4]{Linck18b}, $\lsub{\psi}j\cO{H}j \cong j\cO{H}j$ as $(\cO{Q},j\cO{H}j)$-bimodules and $i\cO{G}i_{\varphi} \cong i\cO{G}i$ as $(i\cO{G}i,\cO{P})$-bimodues.
Consequently, $\lsub{\psi}jMi_{\varphi} \cong jMi$ as $(\cO{Q},\cO{P})$-bimodules and thus
\[ \Ind_R^{Q \times P}V \mid jMi. \]
It follows form (\ref{equ-Mi-N0-N1}) that $jMi=jN_0\oplus jN_1$ with $jN_1$ projective as an $(\cO{Q},\cO{P})$-bimodule.
So (\ref{equ-Ind-mid-N0}) follows and thus
\begin{equation}\label{equ-mid-N*-N}
\Ind_{R^\sharp}^{P\times Q} V^* \otimes_{\cO{Q}} \Ind_R^{Q\times P} V \mid N^*_0j \otimes_{\cO{Q}} jN_0.
\end{equation}

Now assume that $\cO{H}c$ is nilpotent.
Then as in \ref{proof-mainthm}, we may assume $\cO{H}c=j\cO{H}j=\cO{Q}$ and thus $jN_0=N_0$.
In particular, (\ref{equ-mid-N*-N}) becomes
\[ \Ind_{R^\sharp}^{P\times Q} V^* \otimes_{\cO{Q}} \Ind_R^{Q\times P} V \mid N^*_0 \otimes_{\cO{Q}} N_0. \]
Combining the above with (\ref{equ-N*N-mid}) gives (\ref{equ-N*N-cong}).
\end{rem}


\end{document}